\documentclass[12pt,reqno]{amsart}
\usepackage{amsmath, amsthm, amssymb, relsize}

\advance \topmargin by -\headheight
\advance \topmargin by -\headsep
     
\evensidemargin \oddsidemargin
\newtheorem{theorem}{Theorem}[section]
\newtheorem{lemma}[theorem]{Lemma}
\newtheorem{corollary}[theorem]{Corollary}

\newtheorem{proposition}[theorem]{Proposition}

\theoremstyle{definition}
\newtheorem{definition}[theorem]{Definition}
\newtheorem{example}[theorem]{Example}

\theoremstyle{remark}

\numberwithin{equation}{section}

\newcommand{\mmod}[1]{\,\,(\text{mod}\,\,#1)}

\def\bfm{{\mathbf m}}

 \def\bfP{{\mathbf P}}

\def\bfv{{\mathbf v}}
\def\bfw{{\mathbf w}}

\def\calB{{\mathcal B}}

\def\calD{{\mathcal D}}

\def\calI{{\mathcal I}}

\def\calM{{\mathcal M}}

\def\calP{{\mathcal P}}
\def\calQ{{\mathcal Q}}
\def\calR{{\mathcal R}} 

\def\calT{{\mathcal T}}
\def\calU{{\mathcal U}}

\def\dbC{{\mathbb C}}\def\dbE{{\mathbb E}}
\def\dbN{{\mathbb N}}\def\dbP{{\mathbb P}}
\def\dbR{{\mathbb R}}\def\dbT{{\mathbb T}}
\def\dbZ{{\mathbb Z}}

\def\gam{{\gamma}} \def\Gam{{\Gamma}}
\def\del{{\delta}}

 \def\Lam{{\Lambda}} \def\Lamtil{{\widetilde \Lambda}}

\def\sig{{\sigma}}

\def\ome{{\omega}} \def\bfome{{\boldsymbol \omega}} 

\def\eps{\varepsilon}

\def\le{\leqslant} \def\ge{\geqslant}

\newcommand{\avE}{\mathop{\vcenter{\hbox{\relsize{+1}$\dbE$}}}}
\def\pt{{\rm pt}} \def\rmX{{\mathrm X}} \def\rmY{{\mathrm Y}}
\def\abar{\overline a} \def\fbar{\overline f} \def\gbar{\overline g}
\def\onetil{{\widetilde 1}} 

\begin{document}
\title[Multiple recurrence]{Multiple recurrence and convergence\\ along the 
primes}
\author[T. D. Wooley]{Trevor D. Wooley}
\address{TDW: School of Mathematics, 
University of Bristol, 
University Walk, Clifton, 
Bristol BS8 1TW, United Kingdom}
\email{matdw@bristol.ac.uk}
\author[T. D. Ziegler]{Tamar D. Ziegler}
\address{TDZ: Department of Mathematics, Technion-Israel Institute of 
Technology, Haifa 32000, Israel}
\email{tamarzr@tx.technion.ac.il}
\thanks{The first author is supported by a Royal Society Wolfson Research Merit 
Award, and the second author by ISF grant 557/08, an Alon Fellowship, and a 
Landau Fellowship.}
\date{}
\begin{abstract} Let $E\subset \dbZ$ be a set of positive upper density. Suppose
 that $P_1,P_2,\ldots ,P_k\in \dbZ[X]$ are polynomials having zero constant 
terms. We show that the set $E\cap (E-P_1(p-1))\cap \ldots \cap (E-P_k(p-1))$ is
 non-empty for some prime number $p$. Furthermore, we prove convergence in $L^2$
 of polynomial multiple averages along the primes. 
\end{abstract}
\maketitle

\section{Introduction} Given a subset $E$ of the integers having positive upper 
density, the set $E-E$ of differences between pairs of elements of $E$ contains 
an element of the shape $p-1$, with $p$ a prime number. This conjecture of 
Erd\H os was proved by means of the Hardy-Littlewood (circle) method by 
S\'ark\"ozy \cite{Sar1978} in a quantitative form which shows that, if $E-E$ 
contains no shifted prime $p-1$, then necessarily
\begin{equation}\label{1.1}
x^{-1}\text{card}(E\cap [1,x])\ll \frac{(\log \log \log x)^3
(\log \log \log \log x)}{(\log \log x)^2}.
\end{equation}
Subsequent improvements, first by Lucier \cite{Luc2008}, and most recently by 
Ruzsa and Sanders \cite{RS2008}, show that the function on the right hand side 
in the conclusion (\ref{1.1}) may be replaced by $\exp (-c(\log x)^{1/4})$, for 
some positive absolute constant $c$. Problems in which one asks for specified 
constellations of differences between successive terms from a sequence of 
elements in $E$, each difference depending on the same shifted prime, have been 
addressed only very recently. Thus, for example, the problem of exhibiting 
non-trivial three term arithmetic progressions from $E$, with common difference 
a shifted prime, was successfully analysed by Frantzikinakis, Host and Kra 
\cite{FHK2007}, with the analogous problem for longer arithmetic progressions 
conditional on the Inverse Conjecture for Gowers Norms formulated by Green and 
Tao \cite{GT2009}. Our goal in this paper is the unconditional resolution of a 
generalisation of these earlier results, an analogue of the Bergelson-Leibman 
theorem \cite{BL1996}, which exhibits a constellation of differences defined by 
given polynomials whenever these polynomials have zero constant terms.\par

In order to describe our conclusions, we must introduce some notation, and this 
we use throughout. We denote by $[N]$ the discrete interval $\{1,\ldots ,N\}$ of
 natural numbers. Also, we write $|X|$ for the cardinality of a finite set $X$, 
and when $X$ is non-empty, we write
$$\avE_{n\in X}F(n)=\frac{1}{|X|}\sum_{n\in X}F(n).$$
Given a set of integers $E$ having positive upper density, and polynomials 
$P_1,\ldots ,P_k\in \dbZ[x]$, we define the return set $R_{P_1,\ldots ,P_k}$ by
$$R_{P_1,\ldots, P_k}=\{ n\in \dbZ:E\cap (E-P_1(n))\cap \ldots \cap (E-P_k(n))\ne 
\emptyset \}.$$
Finally, we write $\dbP$ for the set of prime numbers. It is natural to conjecture 
that return sets defined by polynomials with zero constant terms contain shifted 
primes (see, for example, Conjecture 1.1 of \cite{LP2009}). Our first result 
confirms this conjecture in full generality for the sets $\dbP\pm 1$ of shifted 
primes.

\begin{theorem}\label{theorem1.1}
Let $E$ be a set of integers having positive upper density, and let 
$P_1,\ldots,P_k\in \dbZ[x]$ satisfy the condition that $P_i(0)=0$ $(1\le i\le 
k)$. Then $R_{P_1,\ldots ,P_k}\cap (\dbP+1)\ne \emptyset$ and $R_{P_1,\ldots ,P_k}\cap 
(\dbP-1)\ne \emptyset$.
\end{theorem}

We are also able to establish that polynomial averages converge when restricted 
to the prime numbers.

\begin{theorem}\label{theorem1.2}
Suppose that $X=(X_0,\calB,\mu,T)$ is an invertible measure preserving system. Let 
$f_1,\ldots ,f_k\in L^\infty(X)$, and let $P_1,\ldots ,P_k\in \dbZ[x]$. Then as 
$N\rightarrow \infty$, the averages
$$\avE_{p\in \dbP \cap [N]}\prod_{j=1}^kf_j(T^{P_j(p)}x)$$
converge in $L^2(X)$.
\end{theorem}

The simplest case of Theorem \ref{theorem1.1} is that in which $k=1$ and 
$P_1(n)=n$. As we have already noted in our opening paragraph, this is the 
case that was successfully considered by S\'ark\"ozy \cite{Sar1978} via the 
circle method. The convergence of the averages asserted by Theorem 
\ref{theorem1.2} in this case was apparently first demonstrated by Weirdl 
\cite{Wie1989}, and pointwise convergence has also been established (see 
\cite{Bou1988}, \cite{Wie1988}). In the special case $k=2$ and 
$(P_1(n),P_2(n))=(n,2n)$, the conclusions of Theorems \ref{theorem1.1} and 
\ref{theorem1.2} have been proved unconditionally by Frantzikinakis, Host and 
Kra \cite{FHK2007}, and subject to the truth of the Inverse Conjecture for 
Gowers Norms described in \cite{GT2009}, this work extends also to any positive 
integer $k$ and linear polynomials $P_i(n)=in$ $(1\le i\le k)$\footnote{The 
Inverse Conjecture for the Gowers norm in the case $k=4$ is by now known 
\cite{GTZ2009}, and thus the proof in \cite{FHK2007} extends unconditionally to 
$4$-term arithmetic progressions.}. We note, however, that the full conclusions 
of Theorems \ref{theorem1.1} and \ref{theorem1.2} do not follow from the 
approach in \cite{FHK2007}, even if one is prepared to assume the latter Inverse 
Conjecture. We remark also that Li and Pan \cite{LP2009} have very recently 
established the case $k=1$ of Theorem \ref{theorem1.1} when the set of shifted 
primes is $\dbP-1$ (see Corollary 1.1 of \cite{LP2009}).\par

Rather than attempt to wield control of the prime variable conjecturally made 
available through Gowers norms, we instead seek control of convergence through a
 variable from the set $E$, switching the roles of this variable and the prime. 
Such a strategy, in which for less well controlled aspects of an analysis one 
may crudely count prime variables by inclusion in larger well-behaved subsets of
 the integers, is reasonably familiar to practitioners of the circle method and 
sieve theory. The mechanism which makes this switching of roles effective is the
 use of the local Gowers norms introduced in \cite{TZ2008}. This allows us to 
assume that the set $E$ possesses extra structure, namely a {\it nilstructure}. 
With this information in hand, we are able to apply the recent work of Green 
and Tao \cite{GT2010}, showing that the M\"obius function is orthogonal to 
(polynomial) nilsequences, in combination with the Leibman structure theorem for
 multivariable polynomial averages \cite{Lei2005a} in order to deliver the 
conclusions of Theorems \ref{theorem1.1} and \ref{theorem1.2}.\par

It seems likely that our methods could be adapted to handle modifications of 
Theorems \ref{theorem1.1} and \ref{theorem1.2} in which the polynomials in 
$\dbZ[x]$ are replaced by general integer-valued polynomials. Indeed, even the 
restriction to polynomials having vanishing constant terms might be weakened 
through a modification of the hypotheses of Theorem \ref{theorem1.1} to 
accommodate jointly intersective polynomials (see \cite{BLL2008} for the 
relevant ideas).\par

We have recorded a number of notational and technical preliminaries relating to 
the ergodic theory that we employ in two appendices at the end of this paper. 
Readers not already aficionados of the subject area would be well-advised to 
peruse this material before continuing further. In particular, we take this 
opportunity to emphasise that throughout this paper, whenever we refer to a 
measure preserving system, we implicitly assume this system to be {\it 
invertible}. In section 2 we outline our approach to the central problem of the
 paper. We consider the prime return set $R_{P_1,\ldots ,P_k}$ in section 3, 
providing the details of the proof of Theorem \ref{theorem1.1}. Section 4 is 
devoted to the convergence of polynomial averages restricted to the primes, 
leading to the proof of Theorem \ref{theorem1.2}.\par

The authors are grateful to the referees of this paper for their careful reading,
 detailed comments, and the consequent improvement in our exposition.

\section{Outline of proof}
We begin by considering a $k$-tuple of polynomials $\bfP=(P_1,\ldots ,P_k)$, and
 a set $E$ having upper density exceeding some positive number $\del$. Our first
 step is to translate the question on the prime return set into an ergodic 
theoretic one via Furstenberg's correspondence principle. Thus we replace the 
set $E$ by a measurable set $A$, of measure $\mu(A)>\del$, in a probability 
measure preserving system $\rmX =(X,\calB,\mu,T)$. By a uniform version of the 
Bergelson-Leibman theorem (see Theorem \ref{theorem3.8} below), there is a 
positive number $c(\del)$ with the property that for any natural number $W$, one
 has
$$\lim_{N\rightarrow \infty}\avE_{n\le N}\mu(T^{-P_1(Wn)}A\cap \ldots \cap 
T^{-P_k(Wn)}A)>c(\del).$$
Here, we emphasise that the number $c(\del)$ depends on $\del$, as well as the 
polynomials $P_1,\dots ,P_k$, but is independent of $A$ and $W$.\par

The ordered polynomial system $\calP=\{P_1,\ldots ,P_k\}$ determines, via PET 
induction, the number of steps $l(\calP)$ that one must take, by repeated 
application of the Cauchy-Schwarz inequality, to obtain a parallelepiped system 
of polynomials independent of the parameter $n$. This in turn determines the 
sieve level $R=N^\eta$, also independent of $W$, via the condition 
$\eta<2^{-3-l(\calP)}$. All estimates henceforth depend implicitly on $\del$ 
and $\eta$.\par

In our next step, we take $w$ to be a slowly growing function of $N$, and put
$$W=\prod_{\substack{p<w\\ p\in \dbP}}p.$$
We will sometimes need to fix $w$ (very large) and take $N$ much larger, and 
for this reason it is useful to adopt the following convention concerning 
Landau's $o$-notation within this paper. As usual, when a quantity approaches 
zero as the main parameter $N$ approaches infinity, we shall say that this 
quantity is $o(1)$. We denote by $o_{w\rightarrow \infty}(1)$ any quantity that 
approaches zero as $w \rightarrow \infty$. Finally, we denote by $o_w(1)$ any 
quantity that, with $w$ fixed, approaches zero as $N \rightarrow \infty$.\par

Next, let $b$ be an integer with $(b,W)=1$. Perhaps it is worth noting that, 
when it comes to establishing Theorem \ref{theorem1.2} in section 4, we must 
consider all possible values of $b$. However, for the proof of Theorem 
\ref{theorem1.1} in section 3, it transpires that the only values of $b$ of 
interest are $\pm 1$ (see the discussion surrounding (\ref{3.5a}) below). We 
define the function $\Lamtil_{w,b}(n)$ by putting
$$\Lamtil_{w,b}(n)=\frac{\phi(W)}{W}\log R=\eta \frac{\phi(W)}{W}\log N,$$
when $Wn+b$ is a prime number\footnote{In the detailed account of our argument 
in section 3, we make the additional technical restriction that 
$\Lamtil_{w,b}(n)$ is thus defined only when $n\in [\frac{1}{2}N]$. The 
straightforward complications associated with this constraint are best ignored 
in the present outline.}, and otherwise by putting $\Lamtil_{w,b}(n)=0$. Here, 
as usual, we write $\phi(W)$ for the Euler totient, so that $\phi(W)=
\prod_{p<w}(p-1)$. In \cite{TZ2008}, an enveloping sieve argument is applied 
to show that there exists a function $\nu_{w,b}(n)$ with the property that 
$\Lamtil_{w,b}(n)\le \nu_{w,b}(n)$, so that $\Lamtil_{w,b}$ is pointwise 
bounded by $\nu_{w,b}$, and
$$\| \nu_{w,b}-1\|_{V_{\calP}}=o_{w\rightarrow \infty}(1).$$
Although we defer until later the definition of the norm here, it may be helpful 
to note that it is similar to a Gowers norm, though with shift sizes short with 
respect to $N$, but larger than the sieve level $R$. We remark that our use of 
notation differs from that in \cite{TZ2008}, owing to the simpler nature of the 
polynomials in question, as well as the absence of scaling issues which obviates 
the need for the full structure theorem proved in \cite{TZ2008}.\par

We examine the average
\begin{equation}\label{2.1}
\avE_{n\le N}\Lamtil_{w,b}(n)\mu (T^{-P_1(Wn)}A\cap \ldots \cap T^{-P_k(Wn)}A),
\end{equation}
and make use of the majorant $\nu_{w,b}$ of $\Lamtil_{w,b}$ to compare it to the 
related average
\begin{equation}\label{2.2}
\avE_{n\le N}\eta \mu (T^{-P_1(Wn)}A\cap \ldots \cap T^{-P_k(Wn)}A),
\end{equation}
which we already know to exceed $\eta c(\del)$. Our aim is to show that the 
difference between these averages is $o_{w\rightarrow \infty}(1)$. This we achieve 
in two steps. The parameter $l(\calP)$ determines a factor $Z_{l(\calP)}(\rmX)$ 
having the structure of an $(l(\calP)-1)$-step nilsystem, this system being 
independent of $W$. In the first step, we show that when $f_i$ is orthogonal to 
$Z_{l(\calP)}(\rmX)$ for some index $i$ with $1\le i\le k$, or equivalently, when 
$\pi:\rmX \rightarrow Z_{l(\calP)}(\rmX)$ is the factor map and $\pi_*f_i=0$, then
\[
\avE_{n\le N}\Lamtil_{w,b}(n)\int \prod_{j=1}^kT^{P_j(Wn)}f_j(x)\,d\mu =o_w(1)+
o_{w\rightarrow \infty}(1).
\]
As usual, here and throughout, we write $Tf(x)$ for $f(Tx)$. We then decompose 
the characteristic function on $A$ by means of the trivial relation 
$1_A=\pi^*\pi_*1_A+(1_A-\pi^*\pi_*1_A)$. Then $\pi_*(1_A-\pi^*\pi_*1_A)=0$, and 
thus
\begin{align*}
\avE_{n\le N}&\Lamtil_{w,b}(n)\mu (T^{-P_1(Wn)}A\cap \ldots \cap T^{-P_k(Wn)}A)\\
&=\avE_{n\le N}\Lamtil_{w,b}(n)\int \prod_{j=1}^k(T^{P_j(Wn)}\pi_*1_A(x)) \ d\pi_*
\mu +o_w(1)+o_{w\rightarrow \infty}(1).
\end{align*}
This allows us to reduce to the situation in which the system $\rmX$ is an 
$l(\calP)$-step pro-nilsystem. In fact, technically speaking, we replace $A$ by 
the non-negative function $\pi_*1_A$, which has integral against $\pi_*\mu$ 
exceeding $\del$. We note that the universality of the constant $c(\del)$ 
applies for any such function.\par

We make an additional reduction to the case in which $f$ is defined on a 
nilsystem $(G/\Gam, \calB,\mu, T)$. This is achieved by means of an 
approximation in $L^2$, and is independent of $w$. If this system is 
disconnected, then it can be decomposed into a union of some finite number, $J$,
 of components $\{ \rmX_i\}_{i=1}^J$ having the property that $T^J:\rmX_i
\rightarrow \rmX_i$ is totally ergodic for $1\le i\le J$. 

We now follow the argument of \cite{GT2009}. We replace $\Lamtil_{w,b}(n)$ by the
 function
$$\Lam_{w,b}(n)=\eta \frac{\phi(W)}{W}\Lam(Wn+b),$$
in which $\Lam$ denotes the classical von Mangoldt function. We then decompose 
$\Lam$ by means of a M\"obius identity into the shape $\Lam^\sharp+\Lam^\flat$, 
corresponding to an associated smooth decomposition of the identity function 
$\chi(x)=x$ in the shape $\chi=\chi^\sharp+\chi^\flat$, with $\Lam^\sharp$ 
associated to small divisors and $\Lam^\flat$ associated to large divisors, just 
as in \cite{GT2009}. Observe next that for any Lipschitz function 
$f$, the expression
$$\prod_{j=1}^k(T^{P_j(Wn)}f(x))$$
is a polynomial nilsequence on $(G/\Gam)^k$. As in \cite{GT2009}, we show that 
the contribution arising from the term
$$\avE_{n\le N}\Bigl( \frac{\phi(W)}{W}\Lam^\sharp(Wn+b)-1\Bigr) \prod_{j=1}^k
(T^{P_j(Wn)}f(x))$$
is negligible. The estimate of the contribution arising from the term 
corresponding to $\Lam^\flat$ follows from Theorem 1.1 of \cite{GT2010}, which 
asserts that the M\"obius function is orthogonal to polynomial nilsequences
with bounds that depend only on the degree of the polynomial and not on the 
polynomial itself.\par

Our goal of showing that the averages (\ref{2.1}) and (\ref{2.2}) are 
asymptotically equal is completed by combining the results of the last 
paragraph, and this completes our outline of the proof.

\section{Prime return sets}
Our objective in this section is the proof of Theorem \ref{theorem1.1}. We begin
 with a discussion of the pseudorandom measures employed in the sketch of the 
argument provided in the previous section.

\subsection{Pseudorandom measures} We first define a normalised counting 
function for prime numbers, with a smoothing weight designed to flatten 
distribution across a subset of residue classes. Let $\eta$ be a positive number 
with $\eta<2^{-3-l(\calP)}$, and put $R=N^\eta$. Define the function $\onetil 
:[N]\rightarrow \dbR^+$ by putting $\onetil (x)=1$ when $x\in [\frac{1}{2}N]$, 
 and otherwise by taking $\onetil (x)=0$. In addition, define $\Lamtil_{w,b}:[N]
 \rightarrow \dbR^+$ by setting
\begin{equation}\label{3.1}
\Lamtil_{w,b}(x)=\frac{\phi(W)}{W}\log R,
\end{equation}
when $x\in [{\textstyle{\frac{1}{2}}}N]$ and $Wx+b\in \dbP$, and otherwise by 
taking $\Lamtil_{w,b}(x)=0$. Here, we choose to identify $[\frac{1}{2}N]$ with 
a subset of $\dbZ/N\dbZ$, in the usual manner. We remark that the function 
$\Lamtil_{w,b}(x)$ is a modification of the classical von Mangoldt function 
$\Lam(x)$. The use of $\log R$ in place of $\log N$, as a normalising factor, is
 necessary in order to bound $\Lamtil$ pointwise by the pseudorandom measure 
$\nu$ shortly to be defined. The ratio $\eta$ between $\log R$ and $\log N$ 
reflects the relative density between the primes, and the almost primes 
occurring implicitly within our argument.\par

An application of the Prime Number Theorem in arithmetic progressions with error 
term (see, for example, Corollary 11.21 of \cite{MV2007}) reveals that when $b$ 
and $W$ are coprime, one has
$$|\{ x\in [{\textstyle{\frac{1}{2}}}N]:Wx+b\in \dbP\}|\gg \frac{W}{\phi(W)}
\frac{N}{\log N}.$$
It follows that $\Lamtil_{w,b}$ has relatively large mean, namely
$$\avE_{n\in [N]}\Lamtil_{w,b}\gg \eta .$$

\par Before announcing the key properties of the pseudorandom measure employed 
in our argument, we must record some definitions. The first definition, of a 
measure, comes from Definition 6.1 of \cite{GT2009}.

\begin{definition}\label{definition3.1} A {\it measure} is a non-negative 
function $\nu_w:[N]\rightarrow \dbR^+$ with the total mass estimate
\begin{equation}\label{3.3}
\avE_{n\in [N]}\nu_w=1+o_{w\rightarrow \infty}(1),
\end{equation}
and such that for each positive number $\eps$, one has the crude pointwise bound
 $\nu_w=O_\eps(N^\eps)$.
\end{definition}

Next we define polynomial norms analogous to Gowers norms.

\begin{definition}\label{definition3.2} Let $a$ be a function from $\dbZ$ into 
$\dbC$ supported in $[N]$. When $k$ is a non-negative integer, we define the 
{\it $V_k$-norm} of $a$ to be the quantity $\|a\|_{V_k}$ defined via the 
relation
$$\|a\|_{V_k}^{2^k}=\avE_{n\le N}\avE_{\substack{m_1,\ldots ,m_k\le \sqrt{N}\\ 
m'_1,\ldots ,m'_k\le \sqrt{N}}}\,\prod_{\bfome\in\{0,1\}^k}a^\bfome(n+\bfome\cdot 
\bfm+({\mathbf 1}-\bfome)\cdot \bfm').$$
Here, we write ${\mathbf 1}$ for the vector $(1,1,\ldots ,1)$, and we put 
$a^\bfome=a$ when $\sum_{i=1}^k\ome_i\equiv 0\pmod{2}$, and otherwise we put 
$a^\bfome=\abar$. Also, when $\calP=\{P_1,\ldots ,P_k\}$ is a standard 
polynomial system with parallelepiped order $l(\calP)$, we define the 
$V_\calP$-norm of the function $a$ by $\|a\|_{V_\calP}=\|a\|_{V_{l(\calP)+1}}$.
\end{definition}

Observe that
$$\|a\|_{V_1}^2=\avE_{n\le N}\Bigl|\avE_{m\le \sqrt{N}}a(n+m)\Bigl|^2\ge 0,$$
so that the definition of the $V_k$-norm makes sense when $k=1$. For larger 
values of $k$, such follows from the following lemma, which records two simple 
properties of the $V_k$-norm useful in our subsequent deliberations.

\begin{lemma}\label{lemma3.3} Let $a$ be a function from $\dbZ$ into $\dbC$ 
supported in $[N]$. When $k$ is a non-negative integer and $0<\gam\le 2^k$, 
one has
$$\avE_{m,m'\le \sqrt{N}}\| a(n+m)\abar(n+m')\|_{V_k}^\gam \le 
\|a\|_{V_{k+1}}^{2\gam},$$
with equality when $\gam=2^k$. If, moreover, the function $a$ has the 
property that for each positive number $\eps$, one has the pointwise bound 
$a(n)=O_\eps(N^\eps)$, then
$$\Bigl|\avE_{n\le N}a(n)\Bigl|\le \|a\|_{V_1}+o(1).$$
\end{lemma}

\begin{proof} The first claim follows at once from the definition of the 
$V_k$-norm, since by H\"older's inequality one has
$$\avE_{m,m'\le \sqrt{N}}\|a(n+m)\abar (n+m')\|_{V_k}^\gam \le \Bigl( \avE_{m,m'\le 
\sqrt{N}}\|a(n+m)\abar (n+m')\|_{V_k}^{2^k}\Bigr)^{\gam 2^{-k}},$$
and the expectation within parentheses on the right hand side here is 
equal to
$$\avE_{n\le N}\avE_{\substack{m_0,\ldots ,m_k\le \sqrt{N}\\ m'_0,\ldots ,m'_k
\le \sqrt{N}}}\,\prod_{\bfome\in\{0,1\}^{k+1}}a^\bfome(n+\bfome\cdot \bfm+
({\mathbf 1}-\bfome)\cdot \bfm')=\| a\|_{V_{k+1}}^{2^{k+1}}.$$

\par The final conclusion of the lemma is essentially a consequence of the 
van der Corput lemma, as in the proof of Lemma A.1 of \cite{TZ2008}, though here
 we are more precise and do not restrict to real functions. Observe that, as a 
consequence of our hypotheses concerning $a(n)$, one has
$$\avE_{n\le N}a(n)=\avE_{m\le \sqrt{N}}\avE_{n\le N}a(n+m)+O(N^{\eps-1/2}).$$
Interchanging the order of summation, an application of Cauchy's inequality 
yields
\begin{align*}
\Bigl|\avE_{n\le N}a(n)\Bigr|^2&\le \avE_{n\le N}\avE_{m,m'\le \sqrt{N}}
\prod_{\ome\in\{0,1\}}a^\ome(n+\ome m+(1-\ome)m')+O(N^{2\eps -1/2})\\
&=\|a\|_{V_1}^2+o(1).
\end{align*}
The desired conclusion is now immediate.
\end{proof}

The following theorem is essentially equivalent to Theorem 3.18 of 
\cite{TZ2008}, and demonstrates the existence of a pseudorandom 
majorant\footnote{\,In modern language, the measure whose existence is asserted 
by Theorem \ref{theorem3.3} is described as a {\it pseudorandom measure}, by 
virtue of the property (\ref{3.2a})}.

\begin{theorem}\label{theorem3.3} Let $\calP$ be a standard polynomial system, 
and let $\eta=2^{-3-l(\calP)}$. Then there exists a measure $\nu_{w,b}$
 with the property that the function $\Lamtil_{w,b}$ defined in $(\ref{3.1})$ 
enjoys the pointwise bound $0\le \Lamtil_{w,b}\le \nu_{w,b}$, and further
\begin{equation}\label{3.2a}
\|\nu_{w,b}-1\|_{V_\calP}=o_{w\rightarrow \infty}(1).
\end{equation}
\end{theorem} 

We note that in \cite{TZ2008}, the parameter $w$ is concretely fixed to be of 
order $\log \log \log N$. In present circumstances, meanwhile, we prefer to 
think of $w$ as very (very) large, but constant, since in the ergodic 
convergence results we do not have uniformity in $w$. To clarify the dependence 
on $w$, we use both the notations $o_w(1)$ and $o_{w\rightarrow \infty}(1)$, as 
defined in section 2.

\subsection{Translation to the ergodic world} We open the main thrust of our 
argument by translating the basic question to an ergodic theoretic setting. We 
achieve this goal by means of the Furstenberg Correspondence Principle (see, 
for example, Furstenberg \cite{Fur1981}).

\begin{lemma}\label{lemma3.4}
Let $E$ be a set of positive upper density in $\dbZ$. Then there exists a measure 
preserving system $\rmX=(X,\calB,\mu,T)$, and an element $A$ of $\calB$ with $\mu(A)>0$, 
with the property that when
$$\mu(A\cap T^{-n_1}A\cap \ldots \cap T^{-n_k}A)>0,$$
then
$$E\cap (E-n_1)\cap \ldots \cap (E-n_k)\ne \emptyset .$$
\end{lemma}

Making use of ergodic decomposition, it follows as a corollary of this 
conclusion that in order to prove Theorem \ref{theorem1.1}, it suffices to 
establish the following ergodic theoretic version of this theorem.

\begin{theorem}\label{theorem3.5}
Suppose that $\rmX=(X,\calB,\mu,T)$ is an ergodic measure preserving system, and
 let $P_1,\ldots ,P_k\in \dbZ[x]$ satisfy the condition $P_i(0)=0$ 
$(1\le i\le k)$. In addition, suppose that $A\in \calB$ satisfies the condition 
$\mu(A)>0$. Let
$$S_{P_1,\ldots ,P_k}=\{ n\in \dbZ:\mu(A\cap T^{-P_1(n)}A\cap \ldots \cap T^{-P_k(n)}A)>0
\}.$$
Then
$$S_{P_1,\ldots ,P_k}\cap (\dbP+1)\ne \emptyset\quad \text{and}\quad S_{P_1,\ldots ,P_k}
\cap (\dbP-1)\ne \emptyset .$$
\end{theorem}

As in many other recurrence results, it is easier to show that the set 
$S_{P_1,\ldots ,P_k}\cap (\dbP\pm 1)$ is large than merely showing that it is not 
empty. In particular, it suffices to show that for any integer $b$ with 
$(b,W)=1$, one has
\begin{equation}\label{3.5a}
\avE_{\substack{n\le N\\ Wn+b\in \dbP}}\mu(A\cap T^{-P_1(Wn)}A\cap \ldots \cap 
T^{-P_k(Wn)}A)>0.
\end{equation}
Notice here that we have no useful control over $W$. However, since $(\pm 1,W)=1$, it 
follows from the Siegel-Walfisz theorem (see, for example, Corollary 11.21 of 
\cite{MV2007}) that for large enough values of $N$ and $b=\pm 1$, the expectation in 
(\ref{3.5a}) is taken over a non-empty set. Hence, the lower bound (\ref{3.5a}) is 
sufficient to establish Theorem \ref{theorem3.5}. On the other hand, the set $\dbP-2$ 
is {\it not} a return set for polynomial averages.\par

The next lemma is classical.

\begin{lemma}\label{lemma3.6}
Suppose that $|a_n|<1$ for each integer $n$. Then one has
$$\Bigl|\avE_{\substack{n\le N\\ Wn+b\in \dbP}}a_{Wn+b}-\avE_{n\le N}\frac{\phi(W)}
{W}\Lam(Wn+b)a_{Wn+b}\Bigr|=o_w(1).$$
\end{lemma}

As a consequence of this result, one may replace the average on the left hand 
side of (\ref{3.5a}) by a weighted average, wherein the weights are given by a 
modified von Mangoldt function. This conclusion we summarise in the next lemma.

\begin{lemma}\label{lemma3.7}
Suppose that $\mu(A)>\del$, for some positive number $\del$. Then, in order to establish 
the lower bound $(\ref{3.5a})$, it suffices to show that
$$\avE_{n\le N}\Lamtil_{w,b}(n)\mu(A\cap T^{-P_1(Wn)}A\cap \ldots \cap T^{-P_k(Wn)}A)
\gg_\del 1+o_w(1)+o_{w\rightarrow \infty}(1).$$
Equivalently, writing $1_A(x)$ for the characteristic function of the set $A$, 
it suffices to show that
\begin{equation}\label{3.6}
\avE_{n\le N}\int \Lamtil_{w,b}(n)1_A(x)\prod_{j=1}^k(T^{P_j(Wn)}1_A(x))\,d\mu \gg_\del 
1+o_w(1)+o_{w\rightarrow \infty}(1).
\end{equation}
\end{lemma}

In order confirm (\ref{3.6}), we require two additional results. The first 
treats an analogous situation in which the von Mangoldt weights are absent, 
a quantitative version of the Polynomial Szemer\'edi theorem.

\begin{theorem}\label{theorem3.8} With the notation and assumptions of the 
previous section, suppose that $\del>0$, and let $g:X\rightarrow \dbR$ be any 
function obeying the pointwise bound $0\le g\le 1+o(1)$, together with the 
mean bound $\int_Xg\,d\mu \ge \del -o(1)$. Then we have
$$\avE_{n\le N}\onetil (n)\int_Xg(x)\prod_{j=1}^k(T^{P_j(Wn)}g(x))\,d\mu 
\ge c(\del)-o_w(1),$$
where $c(\del)$ is a positive number depending on $\del$ and $P_1,\ldots ,P_k$, but
 independent of $W$.
\end{theorem}

\begin{proof} This follows from Theorem 3.2 of \cite{TZ2008}.
\end{proof}

We also require the following structure theorem, due to Leibman \cite{Lei2005a},
 identifying nilsystems as characteristic factors for  multivariate polynomial 
multiple averages.

\begin{theorem}\label{theorem3.9} Suppose that $\rmX=(X,\calB,\mu,T)$ is an 
ergodic measure preserving system. Let $Q_1,\ldots ,Q_s\in \dbZ[x_1,\ldots ,x_m]$
 be polynomials. In addition, let $\calQ$ denote $\{Q_1,\ldots ,Q_s\}$. Then 
there exists a factor $\rmY=(Y,\calD,\nu,S)$ of $\rmX$, with $\pi:\rmX
\rightarrow \rmY$ as the factor map, and an integer $d(\calQ)$, such that:
\begin{enumerate}
\item[(i)] the system $\rmY$ has the structure of an inverse limit of 
$d(\calQ)$-step nilsystems, and
\item[(ii)] the average difference
$$\avE_{\bfm\in \calM}\prod_{j=1}^sT^{Q_j(W\bfm)}f_j-\pi^*\avE_{\bfm\in \calM}\prod_{i=1}^s
S^{Q_j(W\bfm)}\pi_*f_j$$
is $o_w(1)$ in $L^2(X)$. Here, we have written $\calM$ for $[M_1]\times \ldots 
\times [M_m]$, and the convergence is as $M_1,\ldots ,M_m\rightarrow \infty$.
\end{enumerate}
\end{theorem}

Note that the rate of convergence in this theorem may depend on $w$. What is 
crucial is that the integer $d(\calQ)$ is independent of $w$.\par

We at last come to the result of this section which does the heavy lifting in 
our argument. This provides a conclusion on orthogonality to nilsystems.

\begin{proposition}\label{proposition3.10} 
Suppose that $\rmX$ is an ergodic measure preserving system. Let $f_1,\ldots 
,f_k\in L^\infty(X)$ be functions satisfying the condition $\|f_j\|_\infty \le L$ 
$(1\le j\le k)$. Then there exists a factor $\rmY$ of $\rmX$, with $\pi:\rmX
\rightarrow \rmY$ as the factor map, and an integer $d(\calP)$, such that
\begin{enumerate}
\item[(i)] the system $\rmY$ has the structure of an inverse limit of 
$d(\calP)$-step nilsystems, and
\item[(ii)] if, for some index $i$, one has $\pi_*f_i=0$, then 
$$\Bigl| \int \avE_{n\le N}\Lamtil_{w,b}(n)\prod_{j=1}^kT^{P_j(Wn)}f_j(x)\,d\mu \Bigr| 
=o_{L,w}(1)+o_{L,w\rightarrow \infty}(1).$$
\end{enumerate}
\end{proposition}

\begin{proof} The expression which we seek to estimate is
$$\calT=\Bigl| \int \avE_{n\le N}\Lamtil_{w,b}(n)\prod_{j=1}^kT^{P_j(Wn)}f_j(x)\,d\mu 
\Bigr|^2.$$
Observe first that, by the invariance of the measure $\mu$ under the action of 
$T$, it follows that for each positive number $M$, one has
$$\calT=\Bigl| \int \avE_{n\le N}\Lamtil_{w,b}(n) \avE_{l\le M}\prod_{j=1}^k
T^{P_j(Wn)+Wl}f_j(x)\,d\mu \Bigr|^2.$$
Consequently, by applying the Cauchy-Schwarz inequality in combination with the 
triangle inequality, one obtains
$$\calT\le \int \Bigl( \avE_{n\le N}\Lamtil_{w,b}(n)\Bigl| \avE_{l\le M}
\prod_{j=1}^kT^{P_j(Wn)+Wl}f_j(x)\Bigr| \Bigr)^2\,d\mu .$$
By Theorem \ref{theorem3.3}, the modified von Mangoldt function 
$\Lamtil_{w,b}(n)$ is pointwise bounded by the pseudorandom majorant 
$\nu_{w,b}(n)$, and hence we may replace the former by the latter in the last 
upper bound for $\calT$. Proceeding first in this way, and then applying the 
Cauchy-Schwarz inequality once again, we deduce that
\begin{align*}
\calT&\le \int \Bigl( \avE_{n\le N}\nu_{w,b}(n)\Bigl| \avE_{l\le M}\prod_{j=1}^k
T^{P_j(Wn)+Wl}f_j(x)\Bigr| \Bigr)^2\,d\mu  \\
&\le \int \Bigl(\avE_{n\le N}\nu_{w,b}(n)\Bigr)\Bigl(\avE_{n\le N}\nu_{w,b}(n)\Bigl|
\avE_{l\le M}\prod_{j=1}^kT^{P_j(Wn)+Wl}f_j(x)\Bigr|^2\Bigr)\,d\mu .
\end{align*}

\par Our goal in the remainder of the proof is to establish that the integral
\begin{equation}\label{3.b}
\calU=\int \avE_{n\le N}\nu_{w,b}(n)\Bigl|\avE_{l\le M}\prod_{j=1}^kT^{P_j(Wn)+Wl}f_j(x)
\Bigr|^2\,d\mu 
\end{equation}
satisfies
\begin{equation}\label{3.c}
\calU=o_{L,w}(1)+o_{L,w\rightarrow \infty}(1).
\end{equation}
Since equation (\ref{3.3}) provides the estimate
$$\avE_{n\le N}\nu_{w,b}(N)=1+o_{w\rightarrow \infty}(1)$$
for the average of the measure $\nu_{w,b}(n)$, it follows from our earlier 
estimate for $\calT$ together with (\ref{3.b}) and (\ref{3.c}) that
$$\calT\le (1+o_{w\rightarrow \infty}(1))\calU=o_{L,w}(1)+o_{L,w\rightarrow \infty}(1),$$
and this suffices to complete the proof of the theorem.\par

We now focus on (\ref{3.b}), expanding the square in the integrand to obtain
\begin{align*}
\calU&=\int \avE_{n\le N}\nu_{w,b}(n)\avE_{l,m\le M}\prod_{j=1}^kT^{P_j(Wn)+Wl}f_j(x)
\prod_{j=1}^kT^{P_j(Wn)+Wm}\fbar_j(x)\,d\mu \\
&=\int \avE_{n\le N}\nu_{w,b}(n)\avE_{l,m\le M}\prod_{j=1}^kT^{P_j(Wn)}(f_j(x)T^{W(m-l)}
\fbar_j(x))\,d\mu .
\end{align*}
Consider the average
$$\int \avE_{n\le N}\avE_{l,m\le M}\prod_{j=1}^kT^{P_j(Wn)}(f_j(x)T^{W(m-l)}\fbar_j(x))
\,d\mu .$$
Take $M$ to be a real number with $M=N^{O(1)}$. In addition, write
$$Q_{2i-1}(n,m,l)=P_i(n)\quad \text{and}\quad Q_{2i}(n,m,l)=P_i(n)+m-l\quad 
(1\le i\le k),$$
and put $\calQ=\{Q_1,\ldots ,Q_{2k}\}$. Let $\rmY$ be the factor supplied by 
Theorem \ref{theorem3.9} associated with $\calQ$. Then if for some $i$ one has 
$\pi_*f_i=0$, then from the latter theorem it follows that the above average is 
$o_{L,w}(1)$.\par

In view of the above discussion, it suffices to show that for any continuous 
bounded functions $g_1,\ldots ,g_k$ with $\|g_i\|_\infty \le L^2$, one has
$$\int \avE_{n\le N} (\nu_{w,b}(n)-1) \prod_{j=1}^kT^{P_j(Wn)}(g_j(x)
T^{W(m-l)}\gbar_j(x))\,d\mu = o_{L,w}(1)+o_{L,w\rightarrow \infty}(1) .$$
We establish the latter by applying PET induction to show that, whenever $a$ 
is a function from $\dbZ$ into $\dbC$ supported in $[N]$, and satisfying 
$a(n)=O_\eps (N^\eps)$ for every $\eps >0$, then one has
\begin{equation}\label{3.g}
\Bigl|\int \avE_{n\le N}a(n)g_0(x)\prod_{j=1}^kT^{P_j(Wn)}g_j(x)\,d\mu \Bigr| \ll_L 
\|a(n)\|_{V_\calP}+o(1).
\end{equation}
The procedure here is very similar to that applied in \cite{TZ2008}, but 
unfortunately it does not fit precisely into the framework of the latter. We 
therefore repeat the process in the present context. The trick is to insert some
 additional averaging by means of a parameter $M$ of order $\sqrt{N}$. An 
important observation, in this context, is that since the polynomials may be 
supposed distinct, with zero constant terms, then the system $\{P_1,\dots ,P_k
\}$ may be reordered in such a way that we obtain a standard system.\par

We first establish the case in which $\calP$ is a standard linear system. 
Thus we suppose that $\calP=\{P_1,\ldots ,P_k\}$ is a standard linear system, 
and prove by induction on $k$ that
$$\Bigl|\int \avE_{n\le N}a(n)g_0(x)\prod_{j=1}^kT^{P_j(Wn)}g_j(x)\,d\mu \Bigr| \ll_L 
\|a(n)\|_{V_{k+1}}+o(1).$$
For $k=1$, we must estimate the absolute value of the integral
$$\calI_1^*=\int \avE_{n\le N}a(n)g_0(x)T^{P_1(Wn)}g_1(x)\,d\mu .$$
We observe first that $\calI_1^*=\calI_1+o_L(1)$, where we have written
$$\calI_1=\avE_{n\le N}\int g_0(x)\avE_{m\le \sqrt{N}}a(n+m)T^{P_1(W(n+m))}g_1(x)\,
d\mu .$$
By the Cauchy-Schwarz inequality, one has
$$\calI_1\ll_L \avE_{n\le N}\Bigl( \int \Bigl| \avE_{m\le \sqrt{N}}a(n+m)
T^{P_1(W(n+m))}g_1(x)\Bigr|^2\,d\mu \Bigr)^{1/2}.$$
By another application of the Cauchy-Schwarz inequality, we obtain the upper 
bound
\begin{align*}\calI_1^2&\ll_L \avE_{n\le N} \int \Bigl| \avE_{m\le \sqrt{N}}a(n+m)
T^{P_1(W(n+m))}g_1(x)\Bigr|^2\,d\mu \\
&=\avE_{n\le N} \int \avE_{m,m'\le \sqrt{N}}a(n+m)\abar(n+m')T^{P_1(W(n+m))}g_1(x)
T^{P_1(W(n+m'))}\gbar_1(x)\,d\mu .
\end{align*}
Consequently, by the triangle inequality,
$$\calI_1^2\ll_L \avE_{m,m'\le \sqrt{N}}\Bigl| \avE_{n\le N}a(n+m)\abar (n+m')\Bigr| 
\Bigl| \int g_1(x)T^{P_1(W(m'-m))}\gbar_1(x)\,d\mu \Bigr| .$$
Thus, on applying Lemma \ref{lemma3.3} and making yet another application of the 
Cauchy-Schwarz inequality, we deduce that
$$\calI_1^4\ll_L \avE_{m,m'\le \sqrt{N}}(\| a(n+m)\abar(n+m')\|_{V_1}+o(1))^2
\ll \|a\|_{V_2}^4+o(1).$$
This confirms the inductive hypothesis when $k=1$.\par

Suppose now that $K>1$, and the inductive hypothesis holds for $k<K$. In this 
case we evaluate the expression
\begin{equation}\label{3.d}
\calI_K^*=\int \avE_{n\le N}a(n)g_0(x)\prod_{j=1}^KT^{P_j(Wn)}g_j(x)\,d\mu .
\end{equation}
As before, we first obtain the relation $\calI_K^*=\calI_K+o_L(1)$, where
$$\calI_K=\avE_{n\le N}\int g_0(x)\avE_{m\le \sqrt{N}}a(n+m)\prod_{j=1}^KT^{P_j(W(n+m))}
g_j(x)\,d\mu .$$
Next, following an application of the Cauchy-Schwarz inequality, we obtain
$$\calI_K\ll_L \avE_{n\le N}\Bigl(\int \Bigl|\avE_{m\le \sqrt{N}}a(n+m)\prod_{j=1}^K
T^{P_j(W(n+m))}g_j(x)\Bigr|^2\,d\mu \Bigr)^{1/2}.$$
A further application of the Cauchy-Schwarz inequality leads to the relation
\begin{align}
\calI_K^2&\ll_L \avE_{n\le N}\int \Bigl|\avE_{m\le \sqrt{N}}a(n+m)\prod_{j=1}^K
T^{P_j(W(n+m))}g_j(x)\Bigr|^2\,d\mu \notag \\
&=\avE_{n\le N}\int \avE_{m,m'\le \sqrt{N}}a(n+m)\abar(n+m')\notag \\
&\ \ \ \ \ \ \ \ \ \times \prod_{j=1}^KT^{P_j(W(n+m))}g_j(x)\prod_{j=1}^K
T^{P_j(W(n+m'))}\gbar_j(x)\,d\mu .\label{3.e}
\end{align}
Next, owing to the invariance of $\mu$ under the action of $T$, we see that
\begin{align*}
\calI_K^2\ll_L \avE_{m,m'\le \sqrt{N}}&\int \avE_{n\le N}a(n+m)\abar(n+m')\\
&\times g_1^{(m,m')}(x)\prod_{j=2}^KT^{(P_j-P_1)(Wn)}g_j^{(m,m')}(x)\,d\mu ,
\end{align*}
where
$$g_j^{(m,m')}(x)=T^{(P_j-P_1)(Wm)}g_j(x)T^{P_j(Wm')-P_1(Wm)}\gbar_j(x)\quad
 (1\le j\le K).$$
As a consequence of the inductive hypothesis, we therefore deduce by means of Lemma 
\ref{lemma3.3} that
$$\calI_K^2\ll_L\avE_{m,m'\le \sqrt{N}}(\|a(n+m)\abar(n+m')\|_{V_K}+o(1))\ll 
\|a(n)\|^2_{V_{K+1}}+o(1).$$
This confirms the inductive hypothesis, for standard linear systems, when 
$k=K$. The inductive hypothesis consequently holds for all standard linear 
systems.\par

We now apply the PET induction scheme so as to reduce the general case to one in
 which the system $\calP$ is standard and linear. We proceed by induction on the
 weight $w(\calP)$ of the polynomial system $\calP$. Suppose that the desired 
conclusion holds for every standard polynomial system $\calP$ with weight 
$w(\calP)<w$. Since we have already established the desired conclusion for every
 standard linear system, we may suppose that $\calP$ is a standard polynomial 
system of weight $w(\calP)=w$ that is non-linear. As in the linear case, we 
begin by inserting some additional averaging over a variable $m$ running over an
 interval of length $\sqrt{N}$. Thus we evaluate the expression
$$\calI_\calP^*=\int \avE_{n\le N}a(n)g_0(x)\prod_{j=1}^kT^{P_j(Wn)}g_j(x)\,d\mu .$$
The argument leading from (\ref{3.d}) to (\ref{3.e}) may now be applied, without
 modification, to show that $\calI_\calP^*=\calI_\calP+o_L(1)$, where
\begin{align*}
\calI_\calP^2&\ll_L \avE_{n\le N}\int \avE_{m,m'\le \sqrt{N}}a(n+m)\abar(n+m')\\
&\ \ \ \ \ \ \ \ \ \times \prod_{j=1}^kT^{P_j(W(n+m))}g_j(x)\prod_{j=1}^kT^{P_j(W(n+m'))}
\gbar_j(x)\,d\mu .
\end{align*}
Next, applying the invariance of $\mu$ under the action of $T$, we find that
\begin{align}
\calI_\calP^2\ll_L \avE_{n\le N}\int \avE_{m,m'\le \sqrt{N}}&a(n+m)\abar(n+m')
g_1(x)\prod_{j=2}^k T^{(P_j-P_1)(W(n+m))}g_j(x)\notag \\
&\times \prod_{j=1}^k T^{P_j(W(n+m'))-P_1(W(n+m))}\gbar_j(x)\,d\mu .\label{3.f}
\end{align}

\par When $q$ is an integer, define the set of polynomials $\calR'_0(q,q')$ by
$$\calR'_0(q,q')=\{ P_i(W(n+q'))-P_1(W(n+q))\}_{1\le i\le k}.$$
Consider the set of polynomials $\calR'=\calR'_0(m,m)\cup \calR'_0(m,m')$, and let 
$$\calR=\{R_1,\ldots ,R_J\}$$
denote the system obtained from $\calR'$ by removing the polynomials in $\calR'$
 of degree zero with respect to $n$. Then by Lemma \ref{lemmaB.1}, the set of 
polynomials $\calR$ has lower weight with respect to $n$ than the set $\calP$, 
so that $w(\calR)<w$. Moreover, the estimate (\ref{3.f}) takes the shape
$$\calI_\calP^2\ll_L \avE_{n\le N}\int h_0(x)\avE_{m,m'\le \sqrt{N}}a(n+m)\abar(n+m')
\prod_{j=1}^J T^{R_j(m,m',n)}h_j(x)\,d\mu ,$$
in which $h_j$ satisfies $\|h_j\|_\infty =O_L(1)$ $(0\le j\le J)$. Therefore, by 
the inductive hypothesis in combination with Lemma \ref{lemma3.3}, we may conclude 
that
$$\calI_\calP^2\ll_L\avE_{m,m'\le \sqrt{N}}(\|a(n+m)\abar(n+m')\|_{V_\calR}+o(1))\ll 
\|a(n)\|^2_{V_\calP}+o(1).$$ 
This confirms the inductive hypothesis for all standard polynomial systems of 
weight $w$, and hence the inductive hypothesis (\ref{3.g}) has now been 
established for all polynomial systems.
\end{proof}

\begin{corollary}\label{corollary3.11} Provided that the lower bound $(3.5)$ 
holds in the special case wherein $\rmX$ is an inverse limit of nilsystems of 
bounded step, then it holds also without restriction.
\end{corollary}

\begin{proof} Let $\rmY$ be the factor supplied by Proposition 
\ref{proposition3.10}, and let $\pi:\rmX \rightarrow \rmY$ be the associated 
projection. Decompose the characteristic function $1_A$ by means of the identity
 $1_A=\pi^*\pi_*1_A+(1_A-\pi^*\pi_*1_A)$. Then one sees that
\begin{align*}
\Bigl| \int \avE_{n\le N}\Lamtil_{w,b}(n)1_A(x)&\prod_{j=1}^k(T^{P_j(Wn)}1_A(x))\,
d\mu \Bigr|\\
=&\,\Bigl|\int \avE_{n\le N}\Lamtil_{w,b}(n)\pi_*1_A(x)\prod_{j=1}^k(T^{P_j(Wn)}
\pi_*1_A(x))\,d\pi_*\mu \Bigr|\\
&\,+o_w(1)+o_{w\rightarrow \infty}(1).
\end{align*}
\end{proof}

We make one further reduction, from an inverse limit of nilsystems to a 
nilsystem proper, and replace $1_A$ with a Lipschitz continuous function. This 
we accomplish by means of a standard approximation argument, obtaining a 
conclusion independent of $w$.

\begin{lemma}\label{lemma3.12} Suppose that $\rmX$ is an ergodic measure 
preserving system. Let $f_1,\ldots ,f_k\in L^\infty(X)$, and suppose that 
$\|f_j\|_\infty \le 1$. If $g\in L^\infty(X)$ satisfies the condition 
$\|f_1-g\|_2<\eps$, then
$$\Bigl| \int \avE_{n\le N}\Lamtil_{b,w}(n)T^{P_1(Wn)}(f_1-g)(x)\prod_{j=2}^k
T^{P_j(Wn)}f_j(x)\,d\mu \Bigr| =O_w(\eps).$$
\end{lemma}

\begin{proof} By applying the triangle inequality in combination with the 
Cauchy-Schwarz inequality, we obtain
\begin{align*}
\Bigl| \int \avE_{n\le N}\Lamtil_{b,w}(n)&T^{P_1(Wn)}(f_1-g)(x)\prod_{j=2}^k
T^{P_j(Wn)}f_j(x)\,d\mu \Bigr| \\
&\le \avE_{n\le N}\Lamtil_{b,w}(n)\Bigl| \int T^{P_1(Wn)}(f_1-g)(x)
\prod_{j=2}^kT^{P_j(Wn)}f_j(x)\,d\mu \Bigr| \\
&\le \avE_{n\le N}\Lamtil_{b,w}(n)\| f_1-g\|_2\prod_{j=2}^k\|f_j\|_\infty
<\eps (1+o_w(1)).
\end{align*}
\end{proof}

We are consequently able to conclude as follows.

\begin{corollary}\label{corollary3.13}
Provided that the lower bound
$$\avE_{n\le N}\int \Lamtil_{w,b}(n)g(x)\prod_{j=1}^k(T^{P_j(Wn)}g(x))
\,d\mu \gg_\del 1+o_w(1)+o_{w\rightarrow \infty}(1)$$
holds in the special case wherein $\rmX$ is a nilsystem, and $g$ is a 
Lipschitz continuous function satisfying the hypotheses of  Theorem 
$\ref{theorem3.8}$, then it holds also without restriction.
\end{corollary}

Now let $X=(G/\Gam,\calB,\mu,T)$ be an ergodic nilsystem, with the 
transformation $T$ being given by $a\in G$, and write $G^0$ for the identity 
component of $G$. If $X$ is disconnected, then $X^0=G^0\Gamma/\Gamma \cong 
G^0/(\Gamma \cap G^0)$ is a connected component of $X$. Since $X$ is compact,
one finds that $X$ is a disjoint union of finitely many translations of $X^0$,
 say $X=\cup_{i=1}^J a^iX^0$, and the nilsystem $X_0=(X^0, a^J)$ has no finite
 factors. We may now assume further, without loss of generality, that the 
system $X_0$ is of the form  $(L/\Lambda,b)$, where $L$ is connected and 
simply connected (see the discussion at the beginning of section 1 of Leibman 
\cite{Lei2010}). In $L$ there is an element $c$ with $c^J=b$. Since $a^i$ 
induces an isomorphism between $X_0$ and $X_i=(a^iX^0,a^J)$, the same holds 
for $X_i$.\par

For any Lipschitz continuous function $f$, and any fixed $x\in X$, the sequence
 $g_{x,w}(n)$, defined by
 \begin{equation}\label{3.hash}
g_{x,w}(n)=f(a^{P_1(Wn)}x)\ldots f(a^{P_k(Wn)}x),
\end{equation}
is a polynomial nilsequence. Note that for a fixed integer $i$ with 
$1\le i\le J$, the set
$$\{(a^{P_1(Wn)}x,\ldots ,a^{P_k(Wn)}x)\}_{n \equiv i\mmod{J}}$$
is contained in a fixed connected component of $(G/\Gam)^k$. Furthermore, if 
$n=mJ+i$, then
$$a^{P_l(Wn)}x=a^{JP'_{l}(Wm)+q}x=b^{P'_{l}(Wm)}a^qx=c^{JP'_{l}(Wm)+q}
c^{-q}a^qx=c^{P_{l}(Wn)}c^{-q}a^q x,$$
where $P'_l$ and $q$ may depend on $i,W,l,J$. Thus $f(a^{P_l(Wn)}x)$ can be 
viewed as a polynomial nilsequence on the  nilmanifold $L/\Lambda$.\par

Now consider
$$g_{x,w,i}(n)=1_{n\equiv i\mmod{J}}f(a^{P_1(Wn)}x)\ldots f(a^{P_k(Wn)}x).$$
The function $1_{n\equiv i\mmod{J}}$ is a $1$-step nilsequence on the torus 
$\dbT=\dbR/\dbZ$. It is defined by the polynomial $g:\dbZ \to \dbR$ given by 
$g(n)=n/J$, and a function $F:\dbT \to [0,1]$ which is Lipschitz, supported 
on a $1/(10J)$ neighborhood of $i/J$, and for which $F(i/J)=1$. Thus 
$g_{x,w,i}$ is a polynomial nilsequence on the product $\dbT$ and a 
connected component of $(G/\Gam)^k$, with new Lipschitz constant that may 
depend also on $J$.\par 

The upshot of the above discussion is that 
$$g_{x,w}(n)=\sum_{i=1}^J g_{x,w,i}(n),$$
and thus $g_{x,w}(n)$ can be viewed as a polynomial nilsequence on a 
nilmanifold $G/\Gamma$, where the group $G$ is connected and simply connected.

\begin{proposition}\label{proposition3.14}
With the notation  and assumptions in the preamble, one has
\begin{equation}\label{3.5}
\Bigl| \avE_{n\le N}(\Lamtil_{w,b}(n)-\eta \onetil(n))g_{x,w}(n)\Bigr|=
o_{w,\|f\|_{\rm {Lip}}}(1).
\end{equation}
\end{proposition}

This proposition is essentially the polynomial version of Proposition 11.3 of 
\cite{GT2009}. We sketch a proof below. One of the main ingredients is the 
following lemma, which is the polynomial version of Proposition 11.2 of 
\cite{GT2009}. Since the proof is essentially the same, we omit it, though we
 note that one could also prove this lemma using Proposition 11.2 of 
 \cite{GT2009} and the fact that a polynomial nilsequence can be viewed as a 
 linear nilsequence on some nilmanifold of larger nilpotence degree. This is 
 shown in Leibman \cite{Lei2005b} in the context of continuous nilsequences. 
 All that would be required is to verify that Leibman's proof is valid for 
 Lipschitz nilsequences, and is independent of $W$.

\begin{lemma}\label{lemma3.15} Let $F(n)$ be a polynomial nilsequence in 
$G/\Gam$ defined by polynomials from the set $\calP=\{P_1,\dots ,P_k\}$, as in 
$(\ref{3.hash})$, and suppose that $F$ has Lipschitz constant $M$. Let $\eps$ 
be a positive number, and suppose that $N\ge 1$. Then there exist integers 
$r(\calP)$ and $t(\calP)$, and a decomposition $F(n)=F_1(n)+F_2(n)$, where 
$F_1$ is an averaged nilsequence on $(G/\Gam)^{t(\calP)}$ with Lipschitz 
constant $O_{M,\eps,G/\Gam}(1)$, and satisfying
$$\|F_1(n)\|_{U^{r(\calP)}[N]_*}\ll_{M,\eps,G/\Gam }1,$$
and where $\|F_2\|_\infty =O(\eps)$.
\end{lemma}

We now replace $\Lamtil_{w,b}(n)$ by
$$\Lam_{w,b}(n)=\eta \frac{\phi(W)}{W}\Lam(Wn+b)\onetil (n),$$
where $\Lam$ is the classical von Mangoldt function. This is permissible for 
averaging purposes in view of the fact that the difference is negligible on 
average. To this end, we follow section 12 of \cite{GT2009}. Define the function
 $\chi:\dbR^+\rightarrow \dbR^+$ by putting $\chi(x)=x$. We decompose $\chi$ 
via the identity $\chi=\chi^\sharp+\chi^\flat$, where $\chi^\sharp$ is a smooth 
function vanishing for $|x|\ge 1$, and $\chi^\flat$ a smooth function vanishing 
for $|x|\le \frac{1}{2}$. This induces a decomposition $\Lam=\Lam^\sharp+
\Lam^\flat$, with 
$$\Lam^\sharp(n)=-\log R\sum_{d|n}\mu(d)\chi^\sharp \Bigl( \frac{\log d}{\log R}
\Bigr)$$
and
$$\Lam^\flat(n)=-\log R\sum_{d|n}\mu(d)\chi^\flat \Bigl( \frac{\log d}{\log R}
\Bigr).$$

\par Recall the definition of $g_{x,w}(n)$, and define $F_x$ by means of the 
relation $F_x(Wn)=g_{x,w}(n)$. Let $\eps>0$ be sufficiently small, and apply 
Lemma \ref{lemma3.15} to decompose $F_x(Wn)$ in the form $F_{x,1}(Wn)+
F_{x,2}(Wn)$, with conditions silently implied by the suffices $1$ and $2$. 
From here, following the argument of \cite{GT2009} in order to accommodate 
the harmless additional factor $\onetil (n)$, one finds that
\begin{align*}
&\Bigl|\avE_{n\le N}\Bigl(\frac{\phi(W)}{W}\Lam^\sharp (Wn+b)-1\Bigr)
\onetil (n)F_{x,1}(Wn)\Bigr|\\
&\ \ \ \ \ \ \ \le \Bigl\| \Bigl(\frac{\phi(W)}{W}\Lam^\sharp (Wn+b)-1\Bigr)
\onetil (n)\Bigr\|_{U^{r(\calP)}[N]}\|F_{x,1}(Wn)\|_{U^{r(\calP)}[N]_*}\\
&\ \ \ \ \ \ \ =o_{w,\eps ,\|f\|_{\rm Lip}}(1).
\end{align*}   
On the other hand, in view of the upper bound $\|F_{x,2}(Wn)\|_\infty <\eps$, 
we see that
\begin{align*}
&\Bigl|\avE_{n\le N}\Bigl(\frac{\phi(W)}{W}\Lam^\sharp (Wn+b)-1\Bigr)\onetil 
(n)F_{x,2}(Wn)\Bigr|\\
&\ \ \ \ \ \ \ \ll \eps \avE_{n\le N}\Bigl| \Bigl(\frac{\phi(W)}{W}\Lam^\sharp 
(Wn+b)-1\Bigr)\onetil (n)\Bigr|\ll_w \eps .
\end{align*}
Taking $\eps$ now to be a positive function of $N$ decreasing to zero 
sufficiently slowly, it follows from the triangle inequality that
\begin{equation}\label{3.rev1}
\Bigl|\avE_{n\le N}\Bigl(\frac{\phi(W)}{W}\Lam^\sharp (Wn+b)-1\Bigr)\onetil (n)
F_{x}(Wn)\Bigr|=o_{w,\|f\|_{\rm Lip}}(1).
\end{equation}

\par For the remaining part of the dissection, we apply Theorem 1.1 of 
\cite{GT2010}. The sequence
$$h_{x,w}(n)=g_{x,w}((n-b)/W)$$
is  a polynomial nilsequence on the same group with the same Lip constant 
(with a polynomial sequence depending on $W$), and in addition is of the same 
degree. Moreover, one has
$$\avE_{n\le N}\Lam^\flat (Wn+b)g_{x,w}(n)\onetil(n)=\avE_{\substack{b<n\le NW+b\\
 n\equiv b\mmod{W}}}\Lam^\flat (n)h_{x,w}(n)\onetil((n-b)/W).$$
The average on the left hand side of (\ref{3.5}) may therefore be successfully 
estimated by showing that
\begin{equation}\label{3.fx1}
\log R\avE_{m\le \frac{1}{2}NW+b}\,\avE_{\substack{d\le (\frac{1}{2}NW+b)/m\\ 
md\equiv b\mmod{W}}}\mu(d)\chi^\flat \Bigl( \frac{\log d}{\log R}\Bigr)
h_{x,w}(md)=o_{w,\|f\|_{\rm Lip}}(1).
\end{equation}

\par Fortunately, Theorem 1.1 of \cite{GT2010} implies a bound of the shape
$$\avE_{\substack{d\le [M]\\ md\equiv b\mmod{W}}}\mu(d)h_{x,w}(md)\ll 
(1+\|f\|_{\rm Lip})(\log (M/W))^{-A},$$
valid for any positive number $A$ and $M\ge 2W$. We note, in particular, that 
this bound is independent of the polynomial sequence (it depends only on the 
degree), and that there is no restriction on the size of the coefficients of the
 latter polynomial. Since the weight $\chi^\flat(x)$ vanishes for $|x|\le 
\frac{1}{2}$, it follows that the average over $d$ in (\ref{3.fx1}) makes no 
contribution when $[(\frac{1}{2}NW+b)/m]<R^{1/2}$. Consequently, since the weight 
$\chi^\flat(x)$ is smooth, we find by partial summation that when $m\in 
[\frac{1}{2}NW+b]$, the inner average on the left hand side of (\ref{3.fx1}) 
is equal to
$$\avE_{\substack{R^{1/2}<d\le (\frac{1}{2}NW+b)/m\\ md\equiv b\mmod{W}}}\mu(d)
\chi^\flat \Bigl( \frac{\log d}{\log R}\Bigr)h_{x,w}(md)\ll_w 
(1+\|f\|_{\rm Lip})(\log R)^{-A}.$$
We therefore deduce that
$$\avE_{m\le \frac{1}{2}NW+b}\,\avE_{\substack{d\le (\frac{1}{2}NW+b)/m\\ 
md\equiv b\mmod{W}}}\mu(d)\chi^\flat \Bigl( \frac{\log d}{\log R}\Bigr)h_{x,w}(md)
\ll_w (1+\|f\|_{\rm Lip})(\log R)^{-A},$$
and, provided that we take $A>1$, this suffices to deliver the estimate claimed 
in (\ref{3.fx1}).\par

The conclusion of Proposition \ref{proposition3.14} is obtained by combining 
the conclusions of (\ref{3.rev1}) and (\ref{3.fx1}). From here, in view of 
Corollary \ref{corollary3.13}, the lower bound 
(\ref{3.6}) follows on noting that by Theorem \ref{theorem3.8}, one has 
$$\avE_{n\le N}\onetil(n)\int 1_A(x)\prod_{j=1}^k(T^{P_j(Wn)}1_A(x))\,d\mu 
\gg_\del 1+o(1).$$
The lower bound (\ref{3.5a}) now follows from Lemma \ref{lemma3.7}, and this 
completes the proof of Theorem \ref{theorem1.1}. 

\section{Convergence of multiple averages along the primes}
In this section we prove the $L^2$ convergence of polynomial multiple averages 
along the primes. Let $f_1,\dots ,f_k$ be bounded functions. Consider the 
averages
$$A_N(x)=\avE_{\substack{p<N\\ \text{$p$ prime}}}\prod_{j=1}^k T^{P_j(p)}f_j(x).$$
We seek to show that the sequence $\{ A_N(x)\}_{N=1}^\infty$ forms a Cauchy 
sequence in $L^2$. Observe first that the sequence $\{ A_N(x)\}_{N=1}^\infty$ is 
Cauchy if and only if the sequence $\{B_N(x)\}_{N=1}^\infty$ is Cauchy, where
$$B_N(x)=\avE_{n<N}\Lam(n)\prod_{j=1}^kT^{P_j(n)}f_j(x).$$
Indeed, independently of the value of $x$, one has
\begin{align*}
|B_N(x)-A_N(x)|&=\Bigl| \avE_{n<N}\Lam(n)\prod_{j=1}^kT^{P_j(n)}f_j(x)-
\avE_{\substack{p<N\\ \text{$p$ prime}}}\prod_{j=1}^kT^{P_j(p)}f_j(x)\Bigr| \\
&=o(1).
\end{align*}

\par It remains now only to show that the sequence $\{ B_N(x)\}_{N=1}^\infty$ is 
Cauchy, and this we achieve by applying a stronger version of Proposition 
\ref{proposition3.10} that we now briefly pause to establish. This may be 
regarded as a result on orthogonality to nilsystems in $L^2$.

\begin{proposition}\label{proposition4.1} Suppose that $\rmX$ is an ergodic 
measure preserving system. Let $f_1,\ldots ,f_k\in L^\infty(\rmX)$, and suppose 
that $\|f_j\|_\infty\le L$ $(1\le j\le k)$. Then there exists a factor $\rmY$ of 
$\rmX$, with $\pi:\rmX \rightarrow \rmY$ the factor map, and an integer 
$d(\bfP)$, with the following properties:
\begin{enumerate}
\item[(i)] the factor $\rmY$ has the structure of an inverse limit of 
$d(\bfP)$-step nilsystems, and
\item[(ii)] if for some index $i$ one has $\pi_*f_i=0$, then, uniformly in $b$, 
one has 
$$\Bigl\| \avE_{n\le N}\Lamtil_{w,b}(n)\prod_{j=1}^kT^{P_j(Wn+b)}f_j(x)\Bigr\|_2
=o_{L,w}(1)+o_{L,w\rightarrow \infty}(1).$$
\end{enumerate}
\end{proposition}

\begin{proof} In order to establish the proposition, it suffices to confirm that
 the expression
$$M=\Bigl\| \avE_{n\le N}\Lamtil_{w,b}(n)\prod_{j=1}^kT^{P_j(Wn+b)}f_j(x)
\Bigr\|_2^2$$
satisfies the asymptotic relation
$$M=o_{L,w}(1)+o_{L,w\rightarrow \infty}(1).$$
But if we write
$$g_n(x)=\prod_{j=1}^kT^{P_j(Wn+b)}f_j(x),$$
then an application of the triangle inequality to the expansion of $M$ yields 
\begin{align}
M&=\avE_{n,m\le N}\Lamtil_{w,b}(n)\Lamtil_{w,b}(m)\int \prod_{j=1}^kT^{P_j(Wn+b)}f_j(x)
\prod_{j=1}^kT^{P_j(Wm+b)}\fbar_j(x)\,d\mu \notag\\
&\le \avE_{n\le N}\Lamtil_{w,b}(n)\Bigl| \avE_{m\le N}\int \Lamtil_{w,b}(m)g_n(x)
\prod_{j=1}^kT^{P_j(Wm+b)}\fbar_j(x)\,d\mu \Bigr| .\label{4.1}
\end{align}
However, the hypotheses of the proposition imply that $\| g_n(x)\|_\infty \le 
L^k$, and so it follows from Proposition \ref{proposition3.10} that, uniformly in $b$, 
one has
$$\Bigl| \avE_{m\le N}\int \Lamtil_{w,b}(m)g_n(x)\prod_{j=1}^kT^{P_j(Wm+b)}\fbar_j(x)
\,d\mu \Bigr| =o_{L,w}(1)+o_{L,w\rightarrow \infty}(1).$$
The desired conclusion now follows on substituting this estimate into 
(\ref{4.1}).
\end{proof} 

We now return to the proof of Theorem \ref{theorem1.2}. Suppose that $f_1,\dots 
,f_k\in L^\infty(X)$, and that $\|f_j\|\le L$ $(1\le j\le k)$. Let $M$ be a large 
natural number, and put $N=2M$. Observe that since the von Mangoldt function 
$\Lambda$ is supported on prime powers, one has
\begin{align*}
B_{WM}(x)&=\avE_{n<WM}\Lam(n)\prod_{j=1}^kT^{P_j(n)}f_j(x)\\
&=\frac{1}{\phi(W)}\sum_{\substack{0\le b< W\\ (b,W)=1}}\avE_{n<M}\frac{\phi(W)}{W}
\Lam(Wn+b)\prod_{j=1}^kT^{P_j(Wn+b)}f_j(x)+o_L(1)\\
&=\frac{1}{\phi(W)}\sum_{\substack{0\le b<W\\ (b,W)=1}}\avE_{n<N}\eta^{-1}
\Lamtil_{w,b}(n)\prod_{j=1}^kT^{P_j(Wn+b)}f_j(x)+o_{L,w}(1).
\end{align*}
By applying the triangle inequality in combination with Proposition 
\ref{proposition4.1}, we obtain
\begin{align*}
\Bigl\| \avE_{n<N}\Lamtil_{w,b}(n)\prod_{j=1}^kT^{P_j(Wn+b)}f_j(x)-
\avE_{n<N}\Lamtil_{w,b}(n)\prod_{j=1}^k&T^{P_j(Wn+b)}\pi^*\pi_*f_j(x)\Bigr\|_2\\
&=o_{L,w}(1)+o_{L,w\rightarrow \infty}(1),
\end{align*}
in which $\pi$ is the projection onto the relevant nilpotent factor supplied by 
Proposition \ref{proposition3.10}. By Proposition \ref{proposition3.14}, 
meanwhile, one has
$$\Bigl\| \avE_{n<N}(\eta^{-1}\Lamtil_{w,b}(n)-\onetil (n))\prod_{j=1}^k
T^{P_j(Wn+b)}\pi^*\pi_*f_j(x)\Bigr\|_2=o_{L,w}(1).$$

\par Consider next the average $C_{b,M}(x)$ defined by
$$C_{b,M}(x)=\avE_{n<M}\prod_{j=1}^kT^{P_j(Wn+b)}f_j(x).$$
It follows from Leibman \cite{Lei2005a} that the sequence 
$\{C_{b,M}(x)\}_{M=1}^\infty$ converges, and is thus a Cauchy sequence. Fix a positive 
number $\eps$. Then whenever $M_1$ and $M_2$ are sufficiently large, one has
$$\|C_{b,M_1}(x)-C_{b,M_2}(x)\|_2\le \eps.$$
Under the same conditions, moreover, it follows from the triangle inequality in 
combination with the conclusions of the previous paragraph that
$$\Bigl\| B_{WM_i}(x)-\frac{1}{\phi(W)}\sum_{\substack{0\le b<W\\ (b,W)=1}}C_{b,M_i}(x)\Bigr
\|_2<\eps \quad (i=1,2).$$
Consequently, again by the triangle inequality, one finds that whenever $M_1$ and 
$M_2$ are sufficiently large, one has $\|B_{WM_1}(x)-B_{WM_2}(x)\|_2\le 3\eps$, and 
thus the sequence $\{ B_{WM}(x)\}_{M=1}^\infty$ is Cauchy. Finally, since for 
$1\le i\le W$, one has
$$B_{WM+i}(x)=B_{WM}(x)+o_L(1),$$
the sequence $\{ B_M(x)\}_{M=1}^\infty$ is also Cauchy. This confirms our earlier 
claim, and thus the proof of Theorem \ref{theorem1.2} is complete.

\appendix
\section{Ergodic theoretic preliminaries}
We take the opportunity here to prepare some of the infrastructure central to 
the ergodic theory employed in the main body of this paper.

\subsection{Measure preserving systems}
We begin by recalling that a {\it measure preserving transformation} on a 
measure space $(X_0,\calB_X,\mu_X)$ is a map $T:X_0\rightarrow X_0$ satisfying 
the property that, for all $B\in \calB$, one has $\mu(T^{-1}B)=\mu(B)$. A {\it 
probability measure preserving system (m.p.s.)} $X$ is a quadruplet 
$(X_0,\calB_X,\mu_X,T)$, where the triple $(X_0,\calB_X,\mu_X)$ is a probability 
measure space, and $T:X_0\rightarrow X_0$ a measure preserving transformation. 
We define the $L^p$ spaces $L^p(X)=L^p(X_0,\calB_X,\mu_X)$ for $1\le p\le \infty$ 
in the usual manner. Thus, in particular, we identify any two functions in 
$L^p(X)$ which agree $\mu$-almost everywhere. If $X_0$ is a point, we write 
$X=\pt$. We will assume throughout this paper that the measure preserving system
 $X$ is {\em regular}, which is to say that $X_0$ is a compact metric space and 
$\calB_X$ consists of all Borel sets in $X$. There is no loss of generality in 
this assumption since any m.p.s $X$ such that  $\calB_X$ is generated by a 
countable set is equivalent to a regular one.\par

A {\it factor map} $\pi_Y^X:X\rightarrow Y$ is a morphism in the category of 
measure preserving systems. A {\it factor} $(Y_0,\calB_Y,\mu_Y,S,\pi_Y^X)$ of a 
system $X=(X_0,\calB_X,\mu_X,T)$ is a measure preserving system 
$Y=(Y_0,\calB_Y,\mu_Y,S)$ together with a factor map $\pi_Y^X:X\rightarrow Y$. In
 these circumstances, the pushforward $(\pi_Y^X)_*\mu_X$ is equal to $\mu_Y$, and
 the relation $\pi_Y^X\circ T=S\circ \pi_Y^X$ holds $\mu_X$-almost everywhere. 
When $f:Y\rightarrow \dbC$ is a measurable map, we write $(\pi_Y^X)^*f:X
\rightarrow \dbC$ for the pullback defined by $(\pi_Y^X)^*f=f\circ \pi_Y^X$. 
Conversely, when $f\in L^2(X)$, we denote by $(\pi_Y^X)_*f\in L^2(Y)$ the 
pushforward of $f$. We then define the conditional expectation of $f$ to $Y$ by
$$\dbE(f|Y)=(\pi_Y^X)^*(\pi_Y^X)_*f\in L^2(X).$$
We say that $f\in L^2(X)$ is {\it $Y$-measurable} when $f=\dbE(f|Y)$, or 
equivalently, when $f=(\pi_Y^X)F$ for some $F\in L^2(Y)$. In circumstances 
wherein $Y_0$ is a point, we say that $Y$ is {\it trivial} and denote $Y$ as 
$\pt$. Thus, for instance, we may write $(\pi_\pt^X)_* f=\int_Xf\, d\mu_X$. When 
there is no ambiguity we write $\pi$ for $\pi_Y^X$. It is convenient when 
confusion is readily avoided to abuse notation by writing $X$ for the system 
$(X_0,\calB_X,\mu_X,T)$, or for the measure space $(X_0,\calB_X,\mu_X)$, or 
simply for the phase space $X_0$.

\subsection{Nilsystems and nilsequences}
A {\it $k$-step nilsystem} $X$ is a measure preserving system $(X_0,\calB_X,
\mu_X,T)$, in which $X_0=G/\Gam$, for some $k$-step nilpotent Lie group $G$ and 
a cocompact lattice $\Gam$, and $\calB_X$ is the Borel $\sig$-algebra, $\mu_X$ 
the Haar measure, and the measure preserving transformation $T:G/\Gam 
\rightarrow G/\Gam$ is given by a rotation by some group element $a\in G$, 
which is to say that $T(g\Gam)=ag\Gam$. A {\it $k$-step (linear) nilsequence} is a 
sequence of the form $\{F(a^nx)\}_{n\in \dbN}$, where $x\in G/\Gam$ and $F:G/\Gam 
\rightarrow \dbR$ is a continuous function. We endow the nilmanifold $G/\Gam$ 
with a smooth Riemannian metric $d$. Let $g:\dbN \to G$. For $h \in \dbN$ we denote 
$\partial_h g(n) = g(n+h)g^{-1}(n)$. A function $g:\dbN \to G$ is called a 
polynomial sequence of degree $<k$ if, for any $h_1,\ldots,h_k\in \dbN$, one has 
$\partial_{h_k}\ldots \partial_{h_1}g(n)\equiv 1_G$. A {\it degree $<k$ 
polynomial nilsequence} is a sequence of the form $\{F(g(n)x)\}_{n\in \dbN}$, 
where $x\in G/\Gam$ and $F:G/\Gam \rightarrow \dbR$ is a continuous function, 
and $g$ is a polynomial sequence of degree $<k$. We say that a nilsequence 
$\{F(g(n)x)\}$ has Lipschitz constant $L$ if the function $F$ has Lipschitz 
constant $L$. In circumstances in which the representation of the nilsequence is 
not explicit, we define the Lipschitz constant by taking the infimum over all 
possible representations.\par

We next define the Gowers norms, introduced in Lemma 3.9 of \cite{G2001}. Let $a$ 
be a function from $\dbZ/N\dbZ$ into $\dbC$. When $k$ is a non-negative integer, 
we define the {\it  $U^k[N]$-norm} of $a$ to be the quantity $\|a\|_{U^k[N]}$ 
defined via the relation
$$\|a\|_{U^k[N]}^{2^k}=\avE_{n,m_1,\ldots ,m_k\in \dbZ/N\dbZ}\,
\prod_{\bfome\in\{0,1\}^k}a^\bfome(n+\bfome\cdot \bfm),$$
where $a^\bfome=a$ when $\sum_{i=1}^k\ome_i\equiv 0\pmod{2}$, and otherwise 
$a^\bfome=\abar$. Next, we define the {\it dual norm to the Gowers norm} by means 
of the relation
$$\|F\|_{U^k[N]*}=\sup \Bigl\{ \bigl|\avE_{n\in [N]}f(n)F(n)\bigr|:\|f\|_{U^k[N]}
\le 1\Bigr\}.$$

\par Finally, an {\it averaged $k$-step nilsequence} with Lipschitz constant $M$
 is a function $F(n)$ having the form
$$F(n)=\avE_{i\in I}F_i(a_i^nx_i),$$
where $I$ is a finite index set, and for each $i\in I$, the expression 
$F_i(a_i^nx_i)$ is a bounded $k$-step nilsequence on $G/\Gam$ with Lipschitz 
constant not exceeding $M$.

\section{PET induction}
The notion of PET induction was introduced by Bergelson in \cite{Ber1987} as a 
mechanism for establishing a {\it Polynomial Ergodic Theorem} (or PET) for a 
weakly mixing system. We introduce the framework required to apply PET induction
 in this appendix so as to assist in our exposition elsewhere in this paper.\par

A {\it polynomial system} is a set of polynomials $\calP=\{P_1(n),\ldots ,P_k(n)
\}$, where $P_i(n)\in \dbZ[n]$ $(1\le i\le k)$. The {\it degree} of $\calP$ is 
the maximum of the degrees of the polynomials lying in $\calP$. We define an 
equivalence relation on $\dbZ[n]$ by defining the polynomials $P$ and $Q$ to be 
equivalent when $\deg (P-Q)<\deg P$. We then define the {\it degree} of an 
equivalence class to be the degree of its elements. Any polynomial system 
$\calP$ can be partitioned into equivalence classes. For each positive integer 
$l$, let $w_l$ be the number of classes of degree $l$ in $\calP$. Then the 
{\it weight} $w(\calP)$ of the system $\calP$ is defined to be the vector 
$(w_1,\ldots ,w_{\deg \calP})$. Next we establish an order relation on weight 
vectors. Given two integer vectors $\bfv=(v_1,\ldots ,v_r)$ and $\bfw=
(w_1,\ldots ,w_s)$, we write $\bfv<\bfw$ if either $r<s$, or else $r=s$ and 
there is an index $n$ for which $v_j=w_j$ $(n<j\le r)$ and $v_n<w_n$. Subject to 
this relation, the set of weights of polynomial systems is well-ordered. The PET
 induction is an induction on this well-ordered set.\par

An ordered system $\calP=\{P_1\ldots,P_k\}$ is {\it standard} if $\deg P_j>0$ 
for $1\le i\le k$, $\deg (P_i-P_j)>0$ for $i\ne j$, and in addition $P_1$ has 
minimal degree in $\calP$. The system is {\it linear} if each polynomial in 
$\calP$ is linear. The following lemma shows that standard systems are 
well-behaved with respect to a natural differencing operation.

\begin{lemma}\label{lemmaB.1} Let $\calP=\{P_1,\ldots ,P_k\}$ be an ordered 
polynomial system satisfying the property that $P_1$ has minimal degree in 
$\calP$. Given a positive integer $m$, let $\calQ'_m$ be the system defined by
$$\calQ'_m=\{ P_j(n+m)-P_1(n): 1\le j\le k\},$$
and let $\calQ_m^*$ be the set of polynomials lying in $\calQ'_m\cup \calQ'_0$ 
having degree zero in terms of $n$. Finally, denote by $\calQ=\calQ_m(\calP)$ 
the system obtained  from the set $(\calQ'_m\cup \calQ'_0)\setminus \calQ_m^*$ by
 reordering, if possible, so as to respect the conditions described in the 
preamble. Then
\begin{enumerate}
\item[(i)] when $\calP$ is standard and non-linear, the system $\calQ$ is 
standard of weight strictly smaller than the weight of $\calP$, and
\item[(ii)] when $\calP$ is linear, the system $\calQ$ is of weight strictly 
smaller than the weight of $\calP$, though possibly non-standard.
\end{enumerate}
\end{lemma} 

Note that when $\calP$ is a system of weight $(1)$, then the system 
$\calQ$, associated to $\calP$ by the lemma, is empty. Given a standard 
polynomial system $\calP$, the number of steps of the type described in the 
lemma required to reach the empty system is called the {\it parallelepiped 
degree} of $\calP$, denoted by $l(\calP)$. 

\begin{example} Consider the situation in which $\calP=\{ n^2,n^2+n\}$. Then 
$\calP$ is standard of weight $(0,1)$. The system $\calQ_m(\calP)$ associated to
 $\calP$ by Lemma \ref{lemmaB.1} is the system $\calP_1=
\{ n,2mn+m^2,2mn+n+m^2+m\}$, which is standard of weight $(3)$. A second 
application of the lemma associates the system $\calP_2=\calQ_k(\calP_1)$ to 
$\calP_1$, and this is the system 
$$\{ 2mn-n+m^2,2mn+m^2+m,2mn-n+2km+m^2,2mn+2mk+m^2+m+k\},$$
which is non-standard of weight $(2)$. Another application yields the system
$\calP_3=\calQ_l(\calP_2)$, namely
$$\{ n+m,n+2mk+m+k,n+2ml+m,n+2mk+2ml+m+k\},$$
which is of weight $(1)$. Finally, one last application gives the empty set. 
Thus we may conclude that $l(\calP)=4$.
\end{example}

\bibliographystyle{amsbracket}
\providecommand{\bysame}{\leavevmode\hbox to3em{\hrulefill}\thinspace}

\end{document}